\documentclass{article}
\usepackage{amsmath,amssymb}
\usepackage{gastex}
\usepackage{showkeys}
\usepackage{theorem}
\newtheorem{theorem}{Theorem}
\newtheorem{lemma}[theorem]{Lemma}
\newtheorem{proposition}[theorem]{Proposition}
\newtheorem{corollary}[theorem]{Corollary}
{\theorembodyfont{\rmfamily}%
  \newtheorem{example}[theorem]{Example}
   }
\newenvironment{proof}{\noindent\textit{Proof.}}
{\QED\vskip\theorempostskipamount} 

\def\petitcarre{\vrule height4pt width 4pt depth0pt}
\def\QED{\relax\ifmmode\eqno{\hbox{\petitcarre}}\else{%
  \unskip\nobreak\hfil\penalty50\hskip2em\hbox{}\nobreak\hfil
  \petitcarre
  \parfillskip=0pt \finalhyphendemerits=0\par\smallskip}
  \fi}
\DeclareMathOperator{\Card}{Card}
\DeclareMathOperator{\Index}{index}
\def\u(#1){\underline{#1\!}\,}

\newcommand{\N}{\mathbb{N}}

\renewcommand{\P}{\mathcal P}
\title{Hall sets, Lazard sets and comma-free codes}
\author{Dominique Perrin, Christophe Reutenauer}
\begin{document}
\maketitle
\begin{abstract}
We investigate the relationship between two constructions
of maximal comma-free codes described respectively by Eastman
and by Scholtz and the notions of Hall sets and Lazard sets
introduced in connection with factorizations of free monoids
and bases of free Lie algebras.
\end{abstract}
\tableofcontents
\section{Introduction}
The notion of comma-free code has been introduced by Golomb, Gordon
and Welsh~\cite{GolombGordonWelch1958} after the mention
of their possible role in molecular genetics~\cite{CrickGriffithOrgel1957}. These codes are defined
by a property of nonoverlap which makes the decoding very
simple. The definition implies that the
words of a comma-free code are primitive and that
it cannot contain two distinct conjugate words.
It was conjectured in~\cite{GolombGordonWelch1958}
 that for every odd integer $n$
there exists a comma-free code which is a system of
representatives of the conjugacy classes of words of length $n$.

This conjecture was proved by Eastman~\cite{Eastman1965} and later
Scholtz \cite{Scholtz1969} gave a different construction.
In \cite{BerstelPerrinReutenauer2010} and in \cite{Reutenauer1993}, the construction
of Scholtz was described using concepts related with factorizations
of free monoids and bases of free Lie algebras, namely Hall sets
and Lazard sets. These concepts  introduced initially
by Sch\"utzenberger form a remarkable interaction
of notions from classical algebra, such as free Lie algebras
and notions from information theory such as comma-free codes.
They were studied extensively by  Viennot~\cite{Viennot1978}
who introduced the terms of Lazard and Hall sets.

Recently Knuth \cite{Knuth2015} has incorporated the problem of 
constructing comma-free codes as an example involving techniques important
for backtrack programming. He gives in particular a simplified description
of Eastman's construction in Exercise 32.

The aim of this article is to show that Eastman construction
is also related to Hall and Lazard sets. We prove that there
exist a Lazard set of of words such that for any odd integer $n$,
its words of length $n$ form the comma-free codes obtained
by Eastman's construction (Theorem~\ref{theoremMain}).

Eastman's construction has an advantage over Schotz construction: 
it gives
directly for a comma-free code $X$
constructed by his method a polynomial time algorithm to find the conjugate
of a given primitive word which belongs to $X$. This allows
to perform the decoding in polunomial time. We will show
here that this is also the case for Scholtz construction using
an algorithm due to Melan\c{c}on to find the conjugate
of a primitive word which belongs to a given Hall set.

The paper is organized as follows. In Section~\ref{sectionHall}
we recall the definition and basic properties of Hall sets of trees
and words and of Lazard sets of words.

In Section~\ref{sectionCodes}, we recall some definitions and properties
of codes. We define circular codes and the subfamily of comma-free codes.

In Section~\ref{sectionDips}, we give an account of the notions
of dips and superdips introduced by Knuth~\cite{Knuth2015} to
describe Eastmn's algorithm.

In Section~\ref{sectionEastman}, we describe Eastman's algorithm in
terms of the notions introduced in the preceding section
and prove that it gives a maximal comma-free code of length $n$
 for each odd integer $n$ (Proposition~\ref{propositionEastman}).

In Section~\ref{sectionNewLazard} we show the 
existence of a Lazard set $Z$ such that for every odd integer $n$
the set $Z\cap A^n$ is the result of Eastman's construction.

In the last section, we describe the algorithm of Melan\c{c}on
(see~\cite{Reutenauer1993})
which allows to find the conjugate of a primitive word
which belongs to a Lazard set $Z$. It can be applied to
find the conjugate of a primitive word which belongs to
a comma-free code of the form $Z\cap A^n$. In particular
it gives such an algorithm for the code obtained by Scholtz
construction.

\section{Hall sets and Lazard sets}\label{sectionHall}
We begin by recalling the notions of Hall and Lazard sets. These
sets were introduced in connexion with the description of bases
of free Lie algebras (see~\cite{Reutenauer1993} for the historical background).
\subsection{Hall sets}
The free magma $M(A)$ on the alphabet $A$ is the set of all terms containing the letters and closed under the binary operation $x,y\mapsto (x,y)$.
It can be identified with the set of complete binary trees with
leaves labeled by $A$.

A totally ordered subset $H$ of $M(A)$ containing $A$ 
is called a \emph{Hall set} (of trees) if for any $x,y\in H$, one has
$(x,y)\in H$ if and only if $x<y$ and either $y\in A$
or $y=(z,t)$ with $z\le x$. Moreover, in this case
$x<(x,y)$. An element of $H$ is called a \emph{Hall tree}.

We warn the reader that we follow here the notation of Viennot in~\cite{Viennot1978}.
The notation used  in~\cite{Reutenauer1993}, following
Lothaire~\cite{Lothaire1997}, is symmetrical
and a Hall set $H$ in~\cite{Reutenauer1993} is such that for any 
$x,y\in H$, one has $(x,y)\in H$ if and only $x<y$ and either
$x\in A$ or $x=(v,w)$ with $y\le w$. Moreover, in this case
$(x,y)<y$. To recover the above definition, one has to reverse
the order and to take the mirror image.

The \emph{foliage} of an element $z$ of $M(A)$ is the word $f(z)\in A^*$
defined by $f(a)=a$ if $a\in A$ and $f(x,y)=f(x)f(y)$. Thus the foliage
of $z$ is obtained by erasing the parentheses when $z$ is viewed
as a term and by following the frontier of the tree if $z$ is
viewed as a binary tree.

A \emph{Hall set of words} is the foliage of a Hall set tree. Its elements
are called \emph{Hall words}

Fix a Hall set. By~\cite[Corollary 4.5]{Reutenauer1993}, each Hall word is the foliage
of a unique Hall tree.

Two words $x,y$ are \emph{conjugate} if they have the form $x=uv$ and $y=vu$.
Since a conjugate of a word is just cyclic shift,
conjugacy is an equivalence on words.
A word $x$ is \emph{primitive} if it is not a power of another
word. The conjugacy class of a primitive word is made of primitive words
and a primitive word has $|x|$ distinct conjugates (see~\cite{Lothaire1997}
for a more detailed presentation of these properties).

A \emph{Lyndon word} is a primitive word which is minimal in its
conjucacy class.

\begin{example}
The set of Lyndon words on $A$ is a Hall set. Indeed, for each Lyndon
word $x$ which is not a letter, its \emph{standard factorization} is the pair
$(y,z)$ such that $x=yz$ where $y$ is the longest proper prefix of $x$
which is a Lyndon word. We then associate to any Lyndon word $x$ a tree
$\pi(x)$ by $\pi(a)=a$ if $a\in A$ and $\pi(x)=(\pi(y),\pi(z))$ if
$(y,z)$ is the standard factorization of $x$. One may then verify
the condition defining a Hall set (see~\cite{Viennot1978} p. 15
or~\cite[Exercise 8.1.4]{BerstelPerrinReutenauer2010}). The set of Lyndon words of length  $4$ on $A=\{a,b\}$
is $\{aaab,aabb,abbb\}$. The corresponding trees are represented in Figure~\ref{figureLyndon4}.
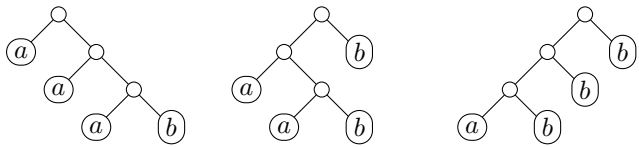
\begin{figure}[hbt]
\centering
\gasset{Nadjust=wh,AHnb=0}
\begin{picture}(80,20)
\put(0,0){
\begin{picture}(20,20)
\node(1)(5,20){}
\node(l)(0,15){$a$}\node(r)(10,15){}
\node(rl)(5,10){$a$}\node(rr)(15,10){}
\node(rrl)(10,5){$a$}\node(rrr)(20,5){$b$}

\drawedge(1,l){}\drawedge(1,r){}
\drawedge(r,rl){}\drawedge(r,rr){}
\drawedge(rr,rrl){}\drawedge(rr,rrr){}
\end{picture}
}
\put(30,0){
\begin{picture}(10,20)
\node(1)(10,20){}
\node(l)(5,15){}\node(r)(15,15){$b$}
\node(ll)(0,10){$a$}\node(lr)(10,10){}
\node(lrl)(5,5){$a$}\node(lrr)(15,5){$b$}

\drawedge(1,l){}\drawedge(1,r){}
\drawedge(l,ll){}\drawedge(l,lr){}
\drawedge(lr,lrl){}\drawedge(lr,lrr){}
\end{picture}
}
\put(60,0){
\begin{picture}(10,20)
\node(1)(15,20){}
\node(l)(10,15){}\node(r)(20,15){$b$}
\node(ll)(5,10){}\node(lr)(15,10){$b$}
\node(lll)(0,5){$a$}\node(llr)(10,5){$b$}

\drawedge(1,l){}\drawedge(1,r){}
\drawedge(l,ll){}\drawedge(l,lr){}
\drawedge(ll,lll){}\drawedge(ll,llr){}
\end{picture}
}
\end{picture}
\caption{The $3$ Lyndon trees of degree $4$}\label{figureLyndon4}
\end{figure}
\end{example}
\subsection{Lazard sets}
We denote $A^{\le n}=\varepsilon\cup A\cup \ldots\cup A^n$ the
set of words of length at most $n$.

A totally ordered set $Z$ of words on the alphabet $A$
is called a \emph{Lazard set} if the following holds. For any
integer $n\ge 1$,  denote the set $Z\cap A^{\le n}=\{z_1,z_2,\ldots,z_k\}$
with 
\begin{equation}
z_1<z_2<\ldots<z_k.\label{equationz_i}
\end{equation} For $1\le i\le k$, let $Z_i$
be the sequence of sets defined by 
$Z_1=A$ and for $1\le i\le k$,  
\begin{equation}
Z_{i+1}=z_i^*(Z_i\setminus z_i).\label{equationZ_i}
\end{equation}
Then $z_i\in Z_i$ for $1\le i\le k$
and 
\begin{equation}
Z_{k+1}\cap A^{\le n}=\emptyset.\label{equationZ_{k+1}}
\end{equation}

The same remark on the choice of right or left made for Hall sets
holds for Lazard sets. We follow here the choice of Viennot~\cite{Viennot1978}
and of~\cite{BerstelPerrinReutenauer2010}. The choice made by the second author
in \cite{Reutenauer1993} is symmetrical. The sets $Z_i$ are defined
by $Z_{i+1}=(Z_i\setminus z_i)z_i^*$.

By a result due to Viennot~\cite{Viennot1978}, Hall sets of words
and Lazard sets
coincide (see~\cite[Therem 4.18]{Reutenauer1993}).

The Hall set of trees corresponding to a given Lazard set $Z$ is
obtained via the mapping $\pi: Z\rightarrow M(A)$ defined as follows.
First $\pi(a)=a$ for any letter $a\in A$. Next, let $n\ge 1$,
let $Z\cap A^{\le n}=\{z_1,z_2,\ldots,z_k\}$ with $z_1<z_2<\ldots<z_k$
and let $Z_1,\ldots,Z_n$ be the sequence of prefix codes defined
by~\eqref{equationZ_i}. Let $z\in Z\cap A^{\le n}$. If $z\in A$,
we set $\pi(z)=z$. Otherwise, let $i$ be the 
least index such that $z\in Z_{i+1}$. One
has then $z=z_iy$ with $y\in Z$ and we set
 $\pi(z)=(\pi(z_i),\pi(y))$.

\begin{example}
Let us verify that the set $L$ of Lyndon words satisfies the condition
of the definition of Lazard sets with $n=5$. One has
$L\cap A^{[5]}=\{a,aaaab,aaab,aaabb,aab, aabab,aabb,aabbb,ab,ababb,abb,abbb,abbbb,b\}$. The corresponding sequence of prefix codes is
\begin{eqnarray*}
Z_1&=&\{a,b\},\\
Z_2\cap A^{\le 5}&=&\{aaaab,aaab,aab,ab,b\},\\
Z_3\cap A^{\le 5}&=&\{aaab,aab,ab,b\},\\
Z_4\cap A^{\le 5}&=&\{aaabb, aab,ab,b\},\\
Z_5\cap A^{\le 5}&=&\{aab,ab,b\},\\
Z_6\cap A^{\le 5}&=&\{aabab,aabb,ab,b\},\\
Z_7\cap A^{\le 5}&=&\{aabb,ab,b\},\\
Z_8\cap A^{\le 5}&=&\{aabbb,ab,b\},\\
Z_9\cap A^{\le 5}&=&\{ab,b\},\\
Z_{10}\cap A^{\le 5}&=&\{ababb,abb,b\},\\
Z_{11}\cap A^{\le 5}&=&\{abb,b\},\\
Z_{12}\cap A^{\le 5}&=&\{abbb,b\},\\
Z_{13}\cap A^{\le 5}&=&\{abbbb,b\},\\
Z_{14}\cap A^{\le 5}&=&\{b\},\\
\end{eqnarray*}
\end{example}

The following result is analogous with Proposition 4.1 in~\cite{Reutenauer1993}
which is stated for Hall sets (and requires a total order on $M(A)$).
\begin{proposition}\label{propositionLazard}
Assume that $A^*$ is totally ordered with an order such that
for any $u,v\in A^*$, if $u<v$, then $u<uv$. Then there is
a unique Lazard set $Z$ ordered by the restriction to $Z$ of this order.
\end{proposition}
\begin{proof}
We show for each $n\ge 1$ the existence and uniqueness
of a sequence $(z_i,Z_i)$ of pairs of a word $z_i$ and a set $Z_i$
such that $z_i\in Z_i$ and satisfying for $1\le i\le k$
conditions \eqref{equationz_i}, \eqref{equationZ_i} and \eqref{equationZ_{k+1}}.

Let $n\ge 1$. Starting with $Z_1=A$, we define for $i\ge 1$
a pair $(z_i,Z_i)$
of a word $z_i$ and a maximal prefix code $Z_i$ as follows.
We choose $z_i\in Z_i$ as the minimal element of $Z_i$
of length at most $n$ and set $Z_{i+1}=z_i^*(Z_i\setminus z_i)$.
We stop when $Z_i\cap A^{\le n}=\emptyset$.

Let $k$ be the first index such that $Z_{k+1}\cap A^{\le n}=\emptyset$.
Then, $z_1<z_2<\ldots<z_k$. Indeed, let $1\le i\le k-1$ and 
set $z_{i+1}=z_i^pz$ with
$z\in Z_i\setminus z_i$. Let us show by induction on $0\le q\le p$ that
 $z_i<z_i^qz$. It is true for $q=0$ since $z_i$ is minimal in $Z_i$.
Next, for $q\ge 1$, set $u=z_i$ and $v=z_i^{q-1}z$. Then
$u<v$ by induction hypothesis and thus $u<uv$ by the hypothesis
made on the order. Thus $z_i<z_i^qz$.

The sequence defined as above for each $n$ is clearly an
initial part of the sequence defined for $n+1$. Thus
there is a unique set $Z$ such that $Z\cap A^{\le n}=\{z_1,\ldots,z_k\}$.
\end{proof}

\section{Comma-free codes}\label{sectionCodes}
In this section, we introduce some basic notions concerning codes (see~\cite{BerstelPerrinReutenauer2010}
for a more detailed presentation).
Let $A$ be an alphabet. A set $X\subset A^+$ is a \emph{code} if
every word in $A^*$ has at most one factorization in words of $X$.
Formally,
for any $x_1,\ldots,x_n$ and $y_1,\ldots y_m$ in $X$ one has
\begin{displaymath}
x_1\cdots x_n=y_1\cdots y_m
\end{displaymath}
 only if $n=m$
and $x_i=y_i$ for $i=1,\ldots,n$. 

A \emph{prefix code} is a set $X\subset A^+$ such that no
word of $X$ is a prefix of another word of $X$.

A \emph{circular code} is a set $X\subset A^+$ such that any word
written on a circle has at most one factorization in words of $X$.
Formally, for any $x_1,\ldots,x_n$ and $y_1,\ldots y_m$ in $X$ 
and $p\in A^*$, $s\in A^+$, one has
\begin{displaymath}
sx_2\cdots x_np=y_1\cdots y_m, \ x_1=ps
\end{displaymath}
 only if $n=m$, $p=\varepsilon$,
and $x_i=y_i$ for $i=1,\ldots,n$. 

It can be shown that a code $X$ is circular if and only if the submonoid $X^*$
satisfies the following condition. For any $u,v\in A^*$, one has
\begin{displaymath}
uv,vu\in X^*\ \Rightarrow u,v\in X^*.
\end{displaymath}

Let $n\ge 1$.
A set $X\subset A^n$ is a \emph{comma-free code} 
if the following property holds. For any $x,y,z\in X$,
if $x$ is a factor of $yz$, then $x=y$ or $x=z$.

One may verify that $X\subset A^n$ is comma-free if and only if
for any $x\in X^+$ and $u,v\in A^*$,
\begin{equation}
uxv\in X^*\Rightarrow u,v\in X^*.\label{eqComma}
\end{equation}

It is clear that a comma-free code is circular. In particular 
a comma-free code of length $n$ contains
only primitive words and at most one element of each conjugacy class
of primitive words of length $n$.

Denote by $\ell_n(k)$ the number of primitive conjugacy classes
of primitive words of length $n$ on an alphabet with $k$ elements.

\begin{theorem}[Eastman]\label{theoremEastman}
For any alphabet $A$ with $k$ letters and for any odd integer $n\ge 1$,
there is a comma-free code $X\subset A^n$ with $\ell_n(k)$ elements.
\end{theorem}
\section{Scholtz algorithm}\label{sectionScholtz}
We reproduce here the construction of Scholtz giving
a proof of Theorem~\ref{theoremEastman} (see~\cite{BerstelPerrinReutenauer2010}
for more details).

We say that a sequence $(x_i,X_i)_{i\ge 1}$ of pais of a word $x_i$ and a set
$X_i$ containing $x_i$ is a \emph{Hall sequence} if it is obtained as follows.
Let $X_1=A$ and let $x_1$ be an arbitrary letter. If $x_i$ and $X_i$ are
already defined, then $X_{i+1}$ is defined by
\begin{equation}
X_{i+1}=x_i^*(X_i\setminus x_i),\label{eqHallSequence}
\end{equation}
and $x_{i+1}$ is an element of $X_{i+1}$ such that $|x_{i+1}|\ge |x_i|$.

To describe Scholtz construction, we build a particular Hall sequence.
For each $i\ge 1$, assuming $(x_1,X_1),\ldots(x_i,X_i)$ already
build, we choose $x_{i+1}$ in $X_i$ as a word of minimal odd length
in $X_i$ and we define $X_{i+1}$ by \eqref{eqHallSequence}.

The following is given in \cite{BerstelPerrinReutenauer2010}
as a proof of Theorem~\ref{theoremEastman}. Set $k=\Card(A)$.

\begin{theorem}
For every odd integer $n$, the set of words of length $n$
in the union of the $X_i$ is a comma-free code with
$\ell_n(k)$ elements.
\end{theorem}

We now relate Scholtz construction with Hall sets.

Let $(X_n,x_n)_{n\ge 1}$ be a Hall sequence and let $U=\bigcup_{n\ge 1} X_n$.
We use for $u\in U$, as in \cite[Section 7.3]{BerstelPerrinReutenauer2010}, the notation
\begin{eqnarray*}
\nu(u)&=&\min\{i\ge 1\mid u\in X_i\}-1,\\
\delta(u)&=&\sup\{i\ge 1\mid u\in X_i\}.
\end{eqnarray*}
 We define
a parenthesis map $\pi:U\rightarrow M(A)$  by $\pi(a)=a$ for $a\in A$
and then inductively $\pi(x)=(\pi(x_n),\pi(y))$ if $\nu(x)=n$ and $x=x_ny$.
 \begin{proposition} The set
$\pi(U)$ is a Hall set for the order
defined by $\pi(u)<\pi(v)$ if $|u|$ is odd and $|v|$ even
or if $|u|,|v|$ have the same parity and $u<v$ for the radix order.
\end{proposition}
Actually, this follows from the genereral
result that Hall sets coincide with Lazard sets, as seen above.

 This can be verified directly as follows.
For any $x,y\in U$, one has
\begin{displaymath}
(\pi(x),\pi(y))\in\pi(U)\Leftrightarrow \nu(y)\le \delta(x)< \delta(y)
\end{displaymath}
and  if $\pi(y)=(\pi(z),\pi(t))$, one has $\nu(y)=\delta(z)$. Thus
$\pi(U)$ is a Hall set.

\section{Dips and superdips}\label{sectionDips}
We now come to the descrition of Eastman's construction, as
presented in~\cite[Exercise 32]{Knuth2015}.

Let $A$ be a totally ordered alphabet with at least two elements. 
Consider the set
\begin{displaymath}
D(A)=\{a_1a_2\cdots a_n\mid a_i\in A, n\ge 2, a_1\ge a_2\ge\ldots\ge a_{n-1}<a_n\}.
\end{displaymath}
The elements of $D(A)$ are called \emph{dips}. Let $S(A)$
be the words formed of a dip of odd length followed by a (possibly empty)
sequence of dips of even length. The elements of $S(A)$
are called \emph{superdips}.

\begin{example}
 For $A=\{a,b,c\}$ and $a<b<c$, we have $D(A)=c^*b^*a^*(ab+ac)+c^*b^*bc$.

\end{example}
Let $X\subset A^*$ be a maximal prefix code. A word $x\in X^*$ is \emph{synchronizing}
if for any $u\in A^*$, one has $ux\in X^*$.
\begin{proposition}
 The set  $D(A)$ is a maximal prefix code such that each word of $D(A)$
of length at least $3$  is synchronizing. 
\end{proposition}
\begin{proof}
 It is clear that $D(A)$ is a prefix code.
To show that it is maximal, assume that $w$
has no prefix in $D(A)$. Then $w=a_1a_2\cdots a_k$
with $a_1\ge \ldots\ge a_k$. If $a_k$ is not 
the largest letter, then $wb\in D(A)$ for a letter $b>a_k$.
Otherwise, $wab$ is in $D(A)$ for $a_k>a<b$
($a,b$ exist since $A$ has at least two elements).

Let $u\in A^*$ and let $x\in D(A)$ be of length
at least $3$. We show that $ux\in D(A)^*$,
which will imply that $x$ is a synchronizing word.
Set $x=pbac$ with $p\in A^*$ and $a,b,c\in A$ and
$b\ge a<c$. Set $ux=yq$ with $y\in D(A)^*$ and $q$ a proper prefix of $D(A)$.
Since $a< c$, the word $ac$ is not an internal factor of $D(A)$
nor a proper prefix of $D(A)$. This implies
that $q=c$ or $q=1$. The first case is not possible because
$y$ would end with $ba$ which is not a suffix of $D(A)$. Thus $ux\in D(A)^*$.
\end{proof}

Note that a word of $D(A)$ of length $2$ may be not synchronizing.
Indeed, for $a<b<c$, we have $bc\in D(A)$ although
$aabc\notin D(A)^*$ and thus $bc$ is not synchronizing.

\begin{proposition}
 The set $S(A)$ is a  code on the alphabet $A$.
\end{proposition}
\begin{proof}
The set $S(A)$ is a  suffix code on the alphabet $D(A)$
and $D(A)$ is a  prefix code on the alphabet $A$.
\end{proof}
Note that $S(A)$ is actually a maximal code on the alphabet $A$.
Indeed, $D(A)$ is a thin maximal prefix code on the
alphabet $A$ and $S(A)$ is a thin maximal suffix code
on the alphabet $D(A)$. Thus $S(A)$ is a thin and complete
code on the alphabet $A$ by \cite[Proposition 2.6.13]{BerstelPerrinReutenauer2010}.
\begin{proposition}\label{propositionConjugate}
 Any primitive word $w$ of odd length
$m\ge 3$ has a conjugate in $D(A)^*$ and thus in $S(A)^*$. More precisely,
the conjugate of $w$ starting after the shortest prefix
of $w^2$ which ends with a dip of length at least $3$
is in $D(A)^*$.
\end{proposition}
\begin{proof}
We first show that $w$ has a conjugate in $D(A)^*$.
 Let $p$ be the shortest prefix of $w^2$ which ends
with a dip of length at least $3$.

 To show the existence of $p$, we consider a factorization
$w=uav$ with $a$ the largest letter among the letters occurring in $w$.
The word $vua$ has at least some ascent (that is, a factor $bc$ with $b<c$),
since $a$ is its largest letter, and since $vua$ is of length at least
$3$ and not a power of $a$. Hence $vua$ has a prefix $q$ which is a dip of length at least $2$. 
Thus $w^2$ has the prefix $p=uaq$
which ends with a dip of length at least $3$.

Assume first that $p$ is a prefix of $w$. Set $w=ps$. Since
$p$ is synchronizing, we have $sp\in D(A)^*$, whence the conclusion.
Assume next that $w$ is a prefix of $p$. Set $p=wr$ and $w=rs$.
Since $p$ is synchronizing, we have $p,wp\in D(A)^*$. Since
$wp=wwr=wrsr=psr$, we have $sr\in D(A)^*$.

 We may now assume that $w\in D(A)^*$. Since $w$ has odd length,
at least one of the dips forming $w$ has odd length. The conjugate
starting before this dip is in $S(A)^*$.
\end{proof}

\begin{example}\label{exampleAbracadabra}
 We illustrate Proposition~\ref{propositionConjugate}
by finding the
conjugate of the word $w=abracadabra$ which is in  $S(A)^*$.

The shortest prefix of $w^2$ which ends with a dip of length at
least $3$ is $p=abrac$. The conjugate of $w$ starting after
$p$ is $ad\ ab\ raab\ rac$ which factorizes in dips as indicated.
Its conjugate $rac\ ad\ ab\ raab$ is in $S(A)$
\end{example}
 
\begin{proposition}\label{propositionOverlap}
 No word of $S(A)$ overlaps nontrivially the
product of two words in $S(A)$, in the sense that for $x,y,z\in S(A)$,
if $x=x_1x_2$ with $y,x_1$ (resp. $z,x_2$) comparable for
the suffix (resp. prefix) order, then $x_1$ or $x_2$ is empty.
\end{proposition}
\begin{proof}

Let $a_1a_2\cdots a_k$ be the last dip of $y$ and let
$b_1b_2\cdots b_\ell$ be the first dip of $z$. Note that
since $z\in S(A)$, $\ell$ is odd. We assume that $x_1,x_2$
are nonempty.

Assume first that $x_1$ is a letter. Then $x_1=a_k$. If $a_k<b_1$, the
first dip of $x$ is $a_kb_1$ and has even length, which is impossible. 
Otherwise,
the first dip of $x$ is $a_kb_1b_2\cdots b_\ell$ and thus has also
even length, a contradiction.

Assume next that $x$ has length $2$. Then $x_1=a_{k-1}a_k$. But then the first dip of $x$
has length $2$, which is again impossible.

Thus $x_1$ has length at least $3$. We distinguish two cases.
\paragraph{Case 1.} Assume that $x_1$ is shorter than $y$. Set $y=ux_1$
(see Figure~\ref{figureSuperdip} on the left).
Since $x_1$ ends with 
$a_{k-1}a_k$ which is not an internal factor of $D(A)$, the first
dip of $x$ is a prefix of $x_1$. Set $x_1=ts$ where $t$ is the
first dip of $x$. Since $t$ has odd length, its length
is at least $3$. Since a dip of length at least $3$ is synchronizing,
and since $y=ux_1=uts$,
we have $ut,s\in D(A)^*$ and thus also $x_1\in D(A)^*$ since $t,s\in D(A)^*$.
 This implies
that $x_2\in D(A)^*$ a contradiction, since the first dip of
$x_2$ is equal to the first dip of $z$ which has  odd length
and since $x$ is a superdip.
\begin{figure}[hbt]
\centering
\gasset{AHnb=0,Nadjust=wh}
\begin{picture}(120,15)(-10,0)
\put(0,0){
\begin{picture}(60,20)
\node(y)(0,0){}\node(z)(30,0){}\node(zr)(50,0){}
\node(u)(0,7){}\node(t)(10,7){}\node(s)(15,7){}\node(sr)(30,7){}
\node(x1)(10,15){}\node(x2)(30,15){}\node(xr)(45,15){}

\drawedge(y,z){$y$}\drawedge(z,zr){$z$}
\drawedge(u,t){$u$}\drawedge(t,s){$t$}\drawedge(s,sr){$s$}
\drawedge(x1,x2){$x_1$}\drawedge(x2,xr){$x_2$}
\drawedge[dash={0.2 0.5}0](x1,t){}
\end{picture}
}
\put(60,0){
\begin{picture}(60,20)
\node(y)(10,0){}\node(z)(30,0){}\node(zr)(50,0){}
\node(u)(0,7){}\node(t)(10,7){}\node(s)(15,7){}\node(sr)(30,7){}
\node(x1)(0,15){}\node(x2)(30,15){}\node(xr)(45,15){}

\drawedge(y,z){$y$}\drawedge(z,zr){$z$}
\drawedge(u,t){$v$}\drawedge(t,s){$t$}\drawedge(s,sr){$s$}
\drawedge(x1,x2){$x_1$}\drawedge(x2,xr){$x_2$}
\drawedge[dash={0.2 0.5}0](y,t){}
\end{picture}
}
\end{picture}
\caption{The two cases}\label{figureSuperdip}
\end{figure}
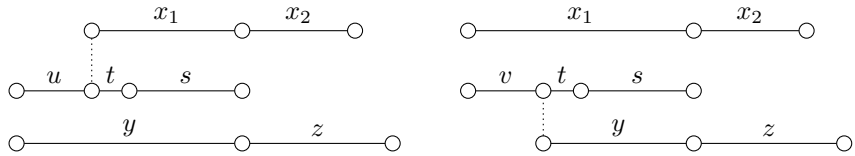
\paragraph{Case 2.} Assume now that $x_1$ is longer than $y$. Set $x_1=vy$
(see Figure~\ref{figureSuperdip} on the right).
Let $t$ be the first dip of $y$. Since $t$ is synchronizing
as in Case 1, we have $t,vt\in D(A)^*$. Set $y=ts$.
Since $y,t\in D(A)^*$, we have $s\in D(A)^*$. Thus $x_2\in D(A)^*$ a contradiction, as in Case 1.
\end{proof}

The following corollary shows that the code $S(A)$ satisfies Condition~\eqref{eqComma}.
\begin{corollary}\label{corollaryComma}
For any $x\in S(A)^+$ and $u,v\in A^*$ such that $uxv\in S(A)^*$,
one has $u,v\in S(A)^*$.
\end{corollary}
\begin{proof}
Set $x=x_1\cdots x_k$ and
$uxv=y_1\cdots y_\ell$ with $y_i\in S(A)$. We assume that
$u,v$ are nonempty. We may assume that $u$
is a prefix of $y_1$ since otherwise we may simplify by $y_1$
on the left. Next, $ux_1$ cannot be a proper prefix of $y_1$
because we would have $u,v\in D(A)^*$ and all dips of $x_1$ would
be even, a contradiction. Thus $x$ overlaps $y_1y_2$ which implies
that $u=y_1$.
\end{proof}
\section{Eastman algorithm}\label{sectionEastman}
 Consider now the radix order on $A^*$,  defined by $u<v$ if $|u|<|v|$
or if   $|u|=|v|$ and $u$ precedes $v$ in lexicographic order. 

 Let $(S_n(A))_{n\ge 0}$ be the sequence of 
maximal  codes on the alphabet $A$ obtained as follows. We start with $S_0(A)=A$. For $n\ge 1$,
let $D_{n}(A)=D(S_{n-1}(A))$ and $S_n(A)=S(S_{n-1}(A))$ where $D,S$ are defined in the previous exercise with
$S_{n-1}(A)$ instead of $A$ as an alphabet, using the order on $S_{n-1}(A)$ induced by the radix
order on $A^*$.

Note that each $S_n(A)$ is formed of words of odd length. Indeed,
this is true for $n=0$ and assuming that it is true for $n-1$, we obtain
the property for $D_n(A)$ using the following lemma whose proof is left to the reader. A \emph{graded alphabet} is a set $A$ with a map $\deg:A\rightarrow \N$
associating to each letter its degree.
The degree of a word is the sum of the degrees of its letters.
\begin{lemma}
Let $A$ be a graded alphabet such that $\deg(a)$ is odd for every $a\in A$.
Then every word of odd length has odd degree.
\end{lemma}

The elements of $D_n(A)$
are called $n$-dips
and those of $S_n(A)$ are called $n$-superdips. Recognizing if a word $w$
is in $D_n(A)^*$ or $S_n(A)ç*$ can be done operating for increasing values of $i=1,\ldots,n$.
Assume that $w\in S_{i-1}(A)^*$. Then one may use a left to right scan of $w$
to write $w=xp$ with $x\in D_{i}(A)^*$ and $p$ a proper prefix of 
$D_{i}(A)^*$. Set $x=x_1\cdots x_k$ with $x_j\in D_i(A)$. Selecting
the blocks beginning with an odd $i$-dip, we obtain
$x=qy_1\cdots y_\ell$ with $y_j\in S_i(A)$. Then $w\in S_i(A)^*$
if and only if $p=q=\varepsilon$.

 Consider the following algorithm to compute a conjugate of a
primitive
word $w\in A^*$ of odd length $m\ge 3$ which is in $S_n(A)$. For successive values
of $n\ge 1$, we perform the following steps.
\begin{enumerate}
\item Let $p$ be the shortest prefix of $w^2$ which ends
with an $n$-dip of length at least $3$ on the alphabet $S_{n-1}(A)$
(such a prefix exists because
the length of $w$ on the alphabet $S_{n-1}(A)$ is at least $3$).  Replace $w$ by its conjugate
starting after $p$.
Now $w\in D_n(A)^*$ (see below).
\item Let $q$ be the first dip of odd length of $w$ (it exists because
$w$ has odd length). 
Replace $w$ by its conjugate starting before $q$. Now $w\in S_n(A)^*$.
\end{enumerate}
We stop when $w$ is in $S_n(A)$.
It follows from  Proposition~\ref{propositionConjugate},
 since $D_n(A)=D(S_{n-1}(A))$,
that the conjugate of
$w$ chosen as in step 1 of the algorithm is in $D_ n(A)^*$.
\begin{example}
 We perform the above algoritm on the word $abracadabra$ (already considered in Example~\ref{exampleAbracadabra}).

 Write periodically the word $abracadabra$ and factorize it in dips.
We obtain
\begin{eqnarray*}
abracadabra\ abracadabra\cdots&=&ab\ rac\ ad\ ab\ raab\ rac\ ad\ ab\ raab\cdots
\end{eqnarray*}
We thus replace $abracadabra$ by its conjugate 
\begin{eqnarray*}
racadabraab&=&rac\ ad\ ab\ raab.
\end{eqnarray*}
 Since
it is a superdip (a dip of length $3$ followed
by dips of lengths $2,2,4$), the algorithm stops.

\end{example}

 We conclude from the above that one has the following result.
\begin{proposition}\label{propositionEastman}
 For every odd integer $m$,
the set of words of length $m$ which belongs to some $S_n(A)$
 is a comma-free code which meets all conjugacy classes
of primitive words of length $m$.
\end{proposition}
\begin{proof}
Set $U=\cup_{n\ge 0}S_n(A)$.
 The set $U\cap A^m$
meets every conjugacy class of of primitive words
of length $m$ because
the algorithm above gives this conjugate. To verify
that it is comma-free, we prove by induction on 
$|x|+|y|+|z|$ that a word $x$ in $U$ does not overlap
nontrivially a product $yz$ of words $y,z\in U$
in the sense that if $x=x_1x_2$ with $x_1$ a proper suffix
of $y$ and $x_2$ a proper prefix of $z$, then $x_1$ or $x_2$
is empty.

The property is true if one of $x,y,z$ is a letter. 

Otherwise, we have $x,y,z\in S_1(A)^*$.
Set $y=ux_1$ and $z=x_2v$. By Corollary \ref{corollaryComma}, 
we have $u,v\in S_1(A)^*$ and thus also
$x_1,x_2\in S_1(A)^*$. Since $S_n(A)=S_{n-1}(S_1(A))$
for all $n\ge 1$,
we may apply the induction hypothesis to the words $x',y',z'$
obtained from $x,y,z$ by considering $S_1(A)$ as a new
alphabet, obtaining the conclusion.
\end{proof}

The following example shows that the comma-free codes obtained
by Eastman algorithm are not the same as those obtained by Scholtz algorithm.
\begin{example}\label{exampleDiff}
Let $A=\{a,b,c\}$, the word
$a^{12}ca^{10}ca^2ca^2ba^2c$ is in $S_2(A)$ since all words
$a^{12}c,a^{10}c,a^2c,a^2b$ are odd length dips and
$a^{12}c\ge a^{10}c\ge a^2c\ge a^2b<a^2c$. Thus it belongs
to the comma-free code of length $33$ obtained by Eastman algorithm.
But it is not in the
code obtained by Scholtz algorithm
which contains the conjugate $a^{10}ca^2ca^2ba^2ca^{12}c$.
\end{example}

\section{A new Lazard set}\label{sectionNewLazard}

The aim of this section is to show that the comma-free
code obtained by Eastman algorithm is a Lazard set
and thus can be obtained by
an elimination method, as in Scholtz algorithm,
although using a different order (which does
not respect length and thus in not defined by a Hall sequence).

Let  $D_n(A)$ be the set of  $n$-dips
as defined in the previous section and let $P_n(A)$ denote
the set of proper prefixes of $D_n(A)$. One has
\begin{displaymath}
A^*\supset D_1(A)^+P_1(A)\supset\cdots\supset D_n(A)^+P_n(A)\supset\cdots
\end{displaymath}
For $w\in A^*$,
the \emph{index} of $w$, denoted $\Index(w)$,
 is the largest integer $n$ such that $w\in D_n(A)^+P_n(A)$.

\begin{example}
The index of $aabaa$ is $1$.
\end{example}

Note that if $x\in S_n(A)$, then $\Index(x)=n$. Indeed, $x\in D_n(A)^+$
and $x$ cannot have a prefix in $D_{n+1}(A)$ since the words
of $D_{n+1}(A)$ have length at least $2$ on the alphabet $S_n(A)$.

We define  the following order on $A^*$. For $x,y\in A^*$,
 we define $x\prec y$ if
\begin{enumerate}
\item[(i)] $x$ has odd length and $y$ has even length, or
\item[(ii)] the lengths of $x$ and $y$ the same parity and
\begin{enumerate}
\item $\Index(x)<\Index(y)$ or 
\item $\Index(x)=\Index(y)$ and $x<y$ for the radix order.
\end{enumerate}
\end{enumerate}
With the objective of applying Proposition~\ref{propositionLazard}, we prove the following
property of this order.
\begin{proposition}
For $x,y\in A^*$, if $x\prec y$, then $x\prec xy$.
\end{proposition}
\begin{proof}
Assume first that  $y$ has even length.
Then the lengths of $x,xy$ 
have the same parity. Assume that $\Index(x)=n$. Set $x=zp$
with $z\in D_n(A)^+$ and $p\in P_n(A)$. Since $D_n(A)$ is a maximal
prefix code, we have $py\in D_n(A)^*P_n(A)$. Thus $xy\in D_n(A)^+P_n(A)$
showing that $\Index (xy)\ge n$. 
If $\Index(xy)>n$, then $x\prec xy$. Otherwise, $x$ is shorter
than $xy$ and thus $x\prec xy$. 

Next, if  $y$ has odd length, then $x$ is also of odd length
and $xy$ has even length, which implies $x\prec xy$.
\end{proof}

Let $Z$ be the Lazard set corresponding to the order $\prec$, (which exists
and is unique by Proposition~\ref{propositionLazard}).

Let $N\ge 1$ be an integer and
set $Z\cap A^{\le N}=\{z_1,z_2,\ldots, z_M\}$ with $z_1\prec z_2\prec \ldots\prec z_M$.
Let
$Z_1,\ldots,Z_{M+1}$ be the sequence of maximal prefix codes 
defined, starting with $Z_1=A$, by  $Z_{n+1}=z_n^*(Z_n\setminus z_n)$.

The \emph{standard factorization} of a word $z$ of $Z\cap A^{\le N}$
which is not a letter
is the pair $(z_n,y)$ such that $z=z_ny$ with $1\le n<M$
and $y\in z_n^*(Z_n\setminus z_n)$. In this way
the Hall tree corresponding to $z$ is $\pi(z)=(\pi(z_n),\pi(y))$.


We need another definition. For some $k\ge 0$,
an element $y\in Z$
is \emph{consistently}  in $D_{k+1}(A)$ if  $y=y_1\cdots y_s$
with $y_i\in S_k(A)\cap Z$, $y_1\ge \ldots\ge y_{s-1}<y_s$,
and  $\pi(y)=(\pi(y_1),\pi(y_2\cdots y_s))$. Thus,
when $y$ is consistently in $D_{k+1}(A)$, its
 standard factorization is given by its
$(k+1)$-dip structure.

Note that any element of $D_1(A)$ is in $Z$ (for large enough $N$)
and is consistently in $D_1(A)$. Indeed, if $y=a_1\cdots a_k$
with $a_i\in A$ and $a_1\ge \ldots\ge a_{k-1}<a_k$,
then $\pi(y)=(a_1\ldots(a_{k-1},a_k))$.

Note also that there can be elements of $Z\cap D_{k+1}(A)$ which are not consistently
in $D_{k+1}(A)$, as shown in the example below.
\begin{example}\label{exampleNotConsistent}
The word $z=aadaabaacab$ is in $D_2(A)$ since
$aad,aab,aacab\in S_1(A)$ and
$aad>aab<aacab$. We have also $z\in Z$ with the
decomposition $\pi(y)=((\pi(aad),(\pi(aab),\pi(aac))),\pi(ab))$. But the two decompositions
do not coincide (see Figure~\ref{figureNonConsistent}
where the tree is only partially developped).
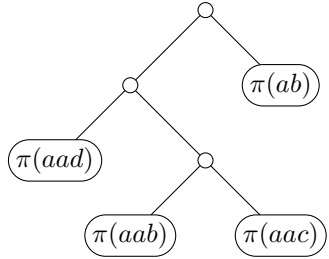
\begin{figure}[hbt]
\centering
\gasset{Nadjust=wh,AHnb=0}
\begin{picture}(50,30)
\node(1)(30,30){}
\node(l)(20,20){}\node(r)(40,20){$\pi(ab)$}
\node(ll)(10,10){$\pi(aad)$}\node(lr)(30,10){}
\node(lrl)(20,0){$\pi(aab)$}\node(lrr)(40,0){$\pi(aac)$}

\drawedge(1,l){}\drawedge(1,r){}
\drawedge(l,ll){}\drawedge(l,lr){}
\drawedge(lr,lrl){}\drawedge(lr,lrr){}
\end{picture}
\caption{The tree $\pi(y)$.}\label{figureNonConsistent}
\end{figure}
\end{example}
We will use two lemmas concerning these notions. The first
one describes a situation where the structure of the Lazard set
coincides with that of the dip and superdip structure.
\begin{lemma}\label{lemma2}
Let $k\ge 0$ and $n\ge 1$ be such that $x\in S_k(A)$ and $y\in  Z_{n+1}$
with $|y|$ odd.
If $y$ is consistently in $D_{k+1}(A)$, then $z_ny$ is in consistently in $D_{k+1}(A)$.
\end{lemma}
\begin{proof}
Set $y=y_1\cdots y_s$ with $y_i\in S_k(A)\cap Z$ and $y_1\ge\ldots\ge y_{s-1}<y_s$. By the hypothesis,
$y_1=z_m$ with $m<n$ and thus $y_1\prec z_n$. Thus
$z_n\ge y_1\ge\ldots\ge y_{s-1}<y_s$, which shows that 
$xy$ is in $D_{k+1}(A)$. It is  consistently
in $D_{k+1}(A)$ since $\pi(z_ny)=(\pi(x),\pi(y))$.
\end{proof}
The next lemma describes a situation where the two structures
are distinct.
\begin{lemma}\label{lemma3}
For $k\ge 0$ and $n\ge 1$, if $x\in S_k(A)$ and $y\in D_m(A)$
with $|y|$ even and $m\le k$, then $xy\in S_k(A)$.
\end{lemma}
\begin{proof}
We prove the statement by induction on $k-m$. It is true if $m=k$
by definition of $S_k(A)$. Assume $m<k$ and set $x=u_1\cdots u_s$
with $u_i\in D_k(A)$.
Set $u_s=v_1\cdots v_t$ with $v_i\in S_{k-1}(A)$ and 
$v_1\ge\ldots\ge v_{t-1}<v_t$. We have $v_ty\in S_{k-1}(A)$
by induction hypothesis. Since $v_t<v_ty$ in the radix order, this implies
$u_sy=v_1\cdots v_{t-1}(v_ty)\in D_k(A)$. 
Since $|y|$ is even, the lengths of $u_s$ and $u_sy$
have the same parity. This implies that $xy=u_1\cdots u_{s-1}(u_sy)$
is in $S_k(A)$.
\end{proof}

\begin{example}
Set $x=aadaabaac$ and $y=ab$ as in Example~\ref{exampleNotConsistent}.
Then, as we have seen, $xy\in D_2(A)$ but the decomposition
of $xy$ in $1$-dips is $(aad,aab,aacab)$ although
$\pi(xy)=((\pi(aad),(\pi(aab),\pi(aac))),\pi(ab))$.
\end{example}

Let $\Delta=\cup_{n\ge 1}D_n(A)$ and $\Sigma=\cup_{n\ge 0}S_n(A)$.
\begin{proposition}\label{propositionInduction}
The words of odd length in $Z$ are in $\Sigma$.
\end{proposition}
\begin{proof}
For a given integer $k\ge 0$, we define
a set  $\P_k$ which is formed of the words $y\in Z$
such that the following holds.
\begin{enumerate}
\item[(i)]  If $|y|$ is odd, then $y$ is
 in $S_k(A)$ or 
consistently in $D_{k+1}(A)$.
\item[(ii)]  If $|y|$ is even,  then $y$ is consistently in $D_{k+1}(A)$
or in $D_m(A)$
with $m\le k$. 
\end{enumerate}
 We prove by induction on $n\ge 1$ that if $n=1$
or if
$z_{n-1}$ has odd length,
one has $Z_{n}\subset \Delta\cup \Sigma$,
and more precisely that, for some integer $k\ge 0$, 
all words of 
$Z_{n}$ are in ${\cal P}_k$.

This is true for $n=1$ with $k=0$. Indeed, 
 $Z_1=A=S_0(A)$. Thus $Z_1\subset \P_0$.

Assume now that $n\ge 2$ and that $z_{n-1}$ has odd length.
Then  by the induction hypothesis
 $Z_{n-1}\subset\P_{k-1}$ for some $k\ge 1$.
Since $z_{n-1}$ has odd length it is either in $S_{k-1}(A)$ or in $D_{k}(A)$.

\paragraph{Case 1}
Assume first that $z_{n-1}\in S_{k-1}(A)$. 
We show that in this case  $Z_{n}\subset \P_{k-1}$. 
We have to prove that for $y\in Z_n\cap \P_{k-1}$, one has $z_{n-1}y\in \P_{k-1}$. 
\begin{enumerate}
\item[1.1 ]Suppose  that
$y$ has odd length. Then  $y$ is in $S_{k-1}(A)$ or consistently in $D_{k}(A)$.
\begin{enumerate}
\item[1.1.1]
If $y\in S_{k-1}(A)$,
then  $z_{n-1}y$
is consistently in $D_{k}(A)$ and thus $z_{n-1}y\in \P_{k-1}$
\item[1.1.2]
If $y\in D_k(A)$, then  $z_{n-1}y$ is consistently in $D_{k}(A)$ by Lemma~\ref{lemma2} and thus $z_{n-1}y\in\P_{k-1}$ .
\end{enumerate}
\item[1.2]
Suppose now that $y$ has even length. Then $y$ is either consistently
in $D_k(A)$ or in $D_m(A)$ for $m<k-1$.
\begin{enumerate}
\item[1.2.1]If $y$
is consistently in $D_{k}(A)$, then $z_{n-1}y$ is consistently in $D_{k}(A)$ 
by Lemma~\ref{lemma2}. 
\item[1.2.2] Otherwise, it is in $D_m(A)$ for $m<k-1$
and thus $z_{n-1}y$ is in $S_{k-1}(A)$ by Lemma~\ref{lemma3}.
\end{enumerate}
In both cases, $z_{n-1}y\in\P_{k-1}$.
\end{enumerate}
\paragraph{Case 2}
Assume now that $z_{n-1}$ is in $D_k(A)$. We show that in this case
$Z_n\subset \P_k$. Since $z_{n-1}\prec z$ for every $z\in Z_{n-1}\setminus z_{n-1}$,
the words of odd length in $Z_{n-1}$ cannot be in $S_{k-1}(A)$. Thus
$Z_{n-1}\subset \P_k$ and we have to prove that for any $y\in Z_n\cap \P_k$,
we have $z_{n-1}y\in \P_k$. The proof is the same as in Case 1 with $k$
instead of $k-1$.
\end{proof}
Using Proposition~\ref{propositionInduction}, we now prove the following result.
\begin{theorem}\label{theoremMain}
 For any odd integer $m$, the set of words
of $\Sigma$ of length $m$ is equal to the set of words of length $m$
in $Z$.
\end{theorem}

\begin{proof}
 We have shown that $Z\cap A^m$ is contained in $\Sigma\cap A^m$.
Since  $Z$ is a Lazard set,
by \cite[Proposition 8.1.10]{BerstelPerrinReutenauer2010},
the family $\{z\mid z\in Z\}$ is a complete factorization of $A^*$.
Thus, by \cite[Corollary 8.1.7]{BerstelPerrinReutenauer2010}, it is a set of representatives of the
primitive conjugacy classes. Since $\Sigma\cap A^m$ is 
comma-free by Proposition~\ref{propositionEastman}, this forces $Z\cap A^m=\Sigma\cap A^m$.
\end{proof}
\section{The Melan\c{c}on algorithm}\label{sectionMelancon}
The last part, taken from~\cite{Reutenauer1993},
 describes an algorithm due to Melan\c{c}on
to find the conjugate of a primitive word which belongs to a Lazard set.
We include it here because it gives an algorithm to
compute the conjugate of a word of even length which belongs to the comma-free
code obtained by Scholtz algorithm.

Let $Z$ be a Lazard set. Consider the following algorithm starting with a 
primitive word
$w=a_1a_2\cdots a_m$ and operating on a sequence $s=(s_1,\ldots,s_n)$
of $n$ elements of $Z$, not all equal. Initially, $s=(a_1,a_2,\ldots,a_m)$.
The main step transformes $s$ as follows. Since not all $s_i$ are equal,
there is an index $i$ with $1\le i\le n$ such that $s_i<s_{i+1}$
(taking the indices cyclically) with $s_i$ minimal among $s_1,\ldots,s_n$.
If $i<n$, change $s$ into $(s_1,\ldots,s_{i-1},s_is_{i+1},s_{i+2},\ldots,s_n)$.
If $i=n$, change $s$ into $(s_ns_1,s_2,\ldots,s_{n-1})$.
The algorithm stops when $n=1$.
\begin{example}
We use the algorithm to find the conjugate of $abracadabra$
which is in the code $S_1$, using the order $\prec$.

 A first sequence of iterations transforms $s=(a,b,r,a,c,a,d,a,b,r,a)$
into $s=(ab,r,ac,ad,ab,r,a)$. At this step, the minimum is the last one
and we obtain $s=(aab,r,ac,ad,ab,r)$. The minimum is now $r$ and thus
we obtain in two steps
$s=(raab,rac,ad,ab)$. The minimum is now $rac$ and we obtain in two steps
$s=(raab,racadab)$. The last one being the minimum, we finally
obtain $s=(racadabraab)$.
\end{example}
 We claim that the result of the algorithm is the conjugate of $w$ which is
in $Z$.

 Let $Z\cap A^{[m]}=\{z_1,z_2,\ldots,z_k\}$
and let $Z_i$ be the sequence defined, as in the definition of Lazard sets,
 by $Z_1=A$ and $Z_{i+1}=z_i^*(Z_i\setminus z_i)$. For $z\in Z$, we denote $\nu(z)=\min\{i\ge 1\mid z\in Z_i\}-1$
and $\delta(z)=\max\{i\ge 1\mid z\in Z_i\}$. Note that
$\delta(z_i)=i$ and that $y<z$ if and only if $\delta(y)<\delta(z)$.
 Then, as for Hall sequences,
for $y,z\in Z$, one
has $yz\in Z$ if and only if $\nu(z)\le\delta(y)<\delta(z)$.
Moreover, in this case, $\nu(yz)=\delta(y)$.

Let us show that any sequence $s=(s_1,\ldots,s_n)$ obtained during
the algorithm is such that all $s_i$ are in $Z$ and for any
$s_i$, either $s_i\in A$ or $\nu(s_i)\le \delta(s_1),\ldots,\delta(s_n)$.
This is true for the initial value of $s$. Next, if we assume that
$s$ has this property and is not constant, let $i$ be such that
$s_i<s_{i+1}$ and $s_i=\min\{s_1,\ldots,s_n\}$. Then,
since $\nu(s_{i+1})\le\delta(s_i)<\delta(s_{i+1})$, we have $s_is_{i+1}\in Z$
and $\nu(s_is_{i+1})=\delta(s_i)$.
Since $\delta(s_i)=\min\{\delta(s_1),\ldots,\delta(s_n)\}$, we conlude
that $\nu(s_is_{i+1})\le\delta(s_1),\ldots,\delta(s_n)$. For the
other $s_j$, we have $\nu(s_j)\le\delta(s_i)$ by induction hypothesis
and thus $\nu(s_j)<\delta(s_is_{i+1})$.
When the algorithm stops, we obtain a word in $Z$.

\bibliographystyle{plain}
\bibliography{lazard}

\end{document}